\documentclass[12pt]{amsart}
\usepackage{amssymb}
\usepackage[all]{xypic}
\usepackage{url}

\theoremstyle{plain}
\newtheorem{theorem}{Theorem}[section]
\newtheorem{lemma}[theorem]{Lemma}
\newtheorem{corollary}[theorem]{Corollary}
\newtheorem{proposition}[theorem]{Proposition}
\newtheorem{conjecture}[theorem]{Conjecture}

\theoremstyle{remark}

\newtheorem{remark}[theorem]{Remark}

\newtheorem*{note*}{Note}
\newtheorem*{remark*}{Remark}
\newtheorem*{example*}{Example}

\theoremstyle{definition}
\newtheorem*{definition*}{Definition}
\newtheorem{definition}[theorem]{Definition}

\newcommand{\Z}{\mathbb{Z}}

\newcommand{\Q}{\mathbb{Q}}

\newcommand{\F}{\mathbb{F}}

\newcommand{\Cl}{\mathrm{Cl}}

\newcommand{\tr}{\mathrm{Tr}}
\newcommand{\OL}{\mathcal{O}_{L}}
\newcommand{\OK}{\mathcal{O}_{K}}
\newcommand{\OM}{\mathcal{O}_{M}}
\newcommand{\ON}{\mathcal{O}_{N}}

\newcommand{\Frakp}{\mathfrak{P}}

\newcommand{\frakp}{\mathfrak{p}}

\newcommand{\frakb}{\mathfrak{b}}

\renewcommand{\mod}{\mathrm{mod}}
\renewcommand{\O}{\mathcal{O}}

\title[On the restricted Hilbert-Speiser and Leopoldt properties]
{On the restricted \\ Hilbert-Speiser and Leopoldt properties}

\author{Nigel P. Byott}
\address{Nigel P. Byott\\
Mathematics Research Institute\\
University of Exeter\\
Harrison \mbox{Building}\\
North Park Road\\
Exeter EX4 4QF\\
U.K.}
\email{N.P.Byott@ex.ac.uk}

\author{James E. Carter}
\address{James E. Carter\\
Department of Mathematics\\
College of Charleston\\
66 George Street\\
Charleston, SC 29424-0001\\
U.S.A.}
\email{carterj@cofc.edu}

\author{Cornelius Greither}
\address{Cornelius Greither\\
Fakult\"at f\"ur Informatik\\
Institut f\"ur theoretische Informatik und Mathematik\\
Universit\"at der Bundeswehr M\"unchen\\
85577 Neubiberg\\
Germany}
\email{cornelius.greither@unibw.de}

\author{Henri Johnston}
\address{Henri Johnston\\ 
St. John's College\\
Cambridge CB2 1TP\\
U.K.}
\email{H.Johnston@dpmms.cam.ac.uk}
\urladdr{http://www.dpmms.cam.ac.uk/$\sim$hlj31}

\date{Version of 9th March 2010}

\begin{document}

\maketitle

\begin{abstract} 
Let $G$ be a finite abelian group. A number field $K$ is called a
Hilbert-Speiser field of type $G$ if, for every tame $G$-Galois
extension $L/K$, the ring of integers $\mathcal{O}_L$ is free as an
$\mathcal{O}_K[G]$-module.  If $\mathcal{O}_L$ is free over the
associated order $\mathcal{A}_{L/K}$ for every $G$-Galois extension
$L/K$, then $K$ is called a Leopoldt field of type $G$.  It is
well-known (and easy to see) that if $K$ is Leopoldt of type $G$, then
$K$ is Hilbert-Speiser of type $G$. We show that the converse does not
hold in general, but that a modified version does hold
for many number fields $K$ (in particular, for $K/\Q$ Galois) when
$G=C_{p}$ has prime order. We give examples with $G=C_p$ to show that
even the modified converse is false in general, and that the
modified converse can hold when the original does not.
\end{abstract}

\section{Introduction} 

Let $L/K$ be a finite abelian extension of number fields with Galois group $G$.
The associated order is defined to be
$
\mathcal{A}_{L/K} := \{ x \in K[G] \, : \, x(\mathcal{O}_L) \subseteq
\mathcal{O}_L \}.
$
In the case $K=\Q$, Leopoldt's Theorem \cite{leopoldt} shows that the
ring of integers $\mathcal{O}_L$ of $L$ is free as a module over
$\mathcal{A}_{L/\Q}$. (A simplified proof of this result can be found
in \cite{lettl}.) More generally, we say that a number field $K$ is
Leopoldt if, for every finite abelian extension $L/K$, the ring of
integers $\mathcal{O}_L$ is free over $\mathcal{A}_{L/K}$ (note that
this differs from the definition of Leopoldt given in \cite{johnston}). Since
$\mathcal{A}_{L/K}=\mathcal{O}_K[G]$ if and only if $L/K$ is tame, 
Leopoldt's Theorem implies the celebrated Hilbert-Speiser
Theorem: Every tame finite abelian extension $L$ of $\mathbb{Q}$ has a
normal integral basis, that is, $\mathcal{O}_L$ is free as a
$\mathbb{Z}[G]$-module. (In this paper, we shall take ``tame''
to mean ``at most tamely ramified''.) A number field $K$ is called a
Hilbert-Speiser field if, for every tame finite abelian extension
$L/K$, the ring of integers $\mathcal{O}_L$ is free over $\mathcal{O}_K[G]$; in particular,
$\mathbb{Q}$ is such a field. The same reasoning as above shows that
if $K$ is Leopoldt then $K$ is Hilbert-Speiser. The converse follows
from Leopoldt's Theorem and the result proven in \cite{grrs} that
$\mathbb{Q}$ is the only Hilbert-Speiser field. Hence we have the
following observation.

\begin{theorem}
Let $K$ be a number field. Then $K$ is a Hilbert-Speiser field if and
only if $K$ is a Leopoldt field.
\end{theorem}

The question arises as to whether a similar result holds when one
fixes the group $G$. 

\begin{definition}\label{definition}
Let $G$ be a finite abelian group and let $K$ be a number field.  Then
$K$ is a \emph{Hilbert-Speiser field of type $G$} if, for every tame
$G$-Galois extension $L/K$, the ring of integers $\mathcal{O}_{L}$ is
free as an $\mathcal{O}_K[G]$-module. Furthermore, $K$ is a
\emph{Leopoldt field of type $G$} if, for every $G$-Galois extension
$L/K$, the ring of integers $\mathcal{O}_{L}$ is free as an
$\mathcal{A}_{L/K}$-module.
\end{definition}

The following conjecture was stated in \cite{carter},
and proved there in the case in which $G=C_{p}$ is cyclic of prime order $p$
and $K$ contains a primitive $p$th root of unity 
$\zeta_{p}$ (see \cite[Theorem 1.2]{carter}).

\begin{conjecture}\label{conjecture}
Let $K$ be a number field and let $G$ be a finite abelian group.
Then $K$ is a Hilbert-Speiser field of type $G$ if and only if $K$ is
a Leopoldt field of type $G$. We denote this statement
by HS-L$(K,G)$.
\end{conjecture}

Miyata \cite{miyata} has investigated the integral Galois module
structure of wildly ramified extensions $L/K$ of number fields of
prime degree. A careful reading of his paper suggests that one should
expect the $\OK$-ideal $\tr_{L/K}(\OL)$ to be a global obstruction to
the freeness of $\OL$ over $\mathcal{A}_{L/K}$. It is therefore
natural to consider the adjusted module ${\tr}_{L/K}(\OL)^{-1}\OL$ in
place of $\OL$ itself. We mention also a related situation where a
similar adjustment is known to be necessary: see for example
\cite[Chapter III, \S3]{taylor}. Let $K$ be a number field, $G$ a
finite abelian group, and $\mathcal{H}$ a Hopf order in $K[G]$.  If
$L$ is a Galois extension of $K$ with group $G$ whose associated order
$\mathcal{A}_{L/K}$ coincides with $\mathcal{H}$, then $\OL$ is
locally free over $\mathcal{H}$. Moreover $\OL$ can be regarded (under
a mild hypothesis on $\mathcal{H}$) as a principal homogeneous space
over the dual Hopf order $\mathcal{H}^*$ to $\mathcal{H}$. There is a
``class invariant'' homomorphism from the group of principal
homogeneous spaces over $\mathcal{H}^*$ into the locally free class
group $\mathrm{Cl}(\mathcal{H})$. This associates to $\OL$ the class
$(\mathcal{H}^*)^{-1} (\OL) = (\tr_{L/K}(\OL)^{-1}\OL)$. On the other
hand, the map which simply associates to $\OL$ its class $(\OL)$ in
$\mathrm{Cl}(\mathcal{H})$ need not be a homomorphism. This lends
further support to the idea that, in general, the class
$(\tr_{L/K}(\OL)^{-1}\OL)$ may be a more natural object to study than
$(\OL)$ itself.

In the light of these observations, we consider in this article the 
following modified versions of Definition \ref{definition} and
Conjecture \ref{conjecture}: 

\begin{definition}\label{modified definition}
Let $G$ be a finite abelian group and let $K$ be a number field.  Then
$K$ is said to satisfy the \emph{modified Leopoldt condition of type
$G$} if, for every $G$-Galois extension $L/K$, the adjusted ring of integers
${\tr}_{L/K}(\OL)^{-1}\OL$ is free as an $\mathcal{A}_{L/K}$-module.
\end{definition}

\begin{conjecture}\label{modified conjecture}
Let $K$ be a number field and let $G$ be a finite abelian group. If
$K$ is a Hilbert-Speiser field of type $G$, then $K$ satisfies the
modified Leopoldt condition of type $G$. We denote this statement by
mHS-L$(K,G)$.
\end{conjecture}

We observe immediately that $\tr_{L/K}(\OL)^{-1} \OL$ is
\emph{locally} free over $\mathcal{A}_{L/K}$ if and only if $\OL$
is. Moreover, $\tr_{L/K}(\OL)^{-1} \OL$ and $\OL$ are isomorphic as
$\mathcal{A}_{L/K}$-modules if and only if $\tr_{L/K}(\OL)$ is a
principal $\OK$-ideal. Thus, for a number field $K$ which is
Hilbert-Speiser of type $G$, there are potentially two ways in which
Conjecture \ref{conjecture} (respectively, Conjecture \ref{modified
conjecture}) might fail. On the one hand, it might fail locally, so
that there is some $G$-Galois extension $L$ of $K$ for which $\OL$ (and hence
also ${\tr}_{L/K}(\OL)^{-1}\OL$) is not even locally free over
$\mathcal{A}_{L/K}$. On the other hand, genuinely global failure may
occur, so that there is a $G$-Galois extension $L$ of $K$ for which $\OL$
(respectively, $\tr_{L/K}(\OL)^{-1} \OL$) is locally free over
$\mathcal{A}_{L/K}$, but not free over $\mathcal{A}_{L/K}$.

We shall show that neither conjecture is true in general, and indeed
that both local and global failure can occur. We shall prove,
nevertheless, that Conjecture
\ref{modified conjecture} holds in many interesting
cases. Specifically, we show first that both conjectures can fail
locally if $G$ is the elementary abelian group $C_2 \times C_2$.

\begin{theorem}\label{C2C2fail}
Let $K$ be a number field such that the ray class group $\Cl_4(\OK)$ modulo
$4 \OK$ is trivial, and such that some prime $\frakp$ of $K$
above $2$ has absolute ramification index at least $3$. 
Let $G=C_2 \times C_2$.
Then $K$ is Hilbert-Speiser of type $G$, but there exists
a $G$-Galois extension $L$ of $K$ such that $\OL$ is not locally free over
$\mathcal{A}_{L/K}$.
\end{theorem}

\begin{corollary} \label{C2C2ex}
HS-L$(K,C_2\times C_2)$ does not hold if $K$ is any of the three real cubic 
fields $\Q(\alpha_i)$, $1 \leq i \leq 3$, where $\alpha_1$,
$\alpha_2$, $\alpha_3$ is a zero of 
$x^3+x^2-3x-1$, $x^3-x^2-5x-1$, $x^3+x^2-5x-3$, respectively. 
\end{corollary}

We shall then specialise to the case where $G=C_{p}$ is cyclic of
prime order $p$, as in \cite{carter}. The above-mentioned work of Miyata
allows us to prove the following key proposition, which in particular
shows that global failure of mHS-L$(K,C_p)$ cannot occur.

\begin{proposition}\label{equivalence}
A number field $K$ satisfies the modified Leopoldt condition of type
$C_p$ if and only if 
\begin{enumerate}
\item $K$ is Hilbert-Speiser of type $C_p$; and
\item for every wildly ramified $C_p$-Galois extension $L/K$, the ring of
integers $\mathcal{O}_L$ is locally free over its associated order
$\mathcal{A}_{L/K}$.
\end{enumerate}
Hence  mHS-L$(K,C_{p})$ holds if and only if either (a) does not hold
or (b) does hold. 
\end{proposition}

With this in mind, we point out the main result of \cite{gh}. Its
proof is based on a detailed analysis of locally free class groups and
ramification indices.

\begin{theorem}\label{real-not-HS}
Let $K$ be a totally real number field and let $p\geq5$ be prime.
Suppose that $K/\Q$ is ramified at $p$. If $p=5$ and
$[K(\zeta_{5}):K]=2$, assume further that there exists a prime
$\mathfrak{p}$ of $K$ above $p$ with absolute ramification index at least $3$. 
Then $K$ is not Hilbert-Speiser of type $C_{p}$.
\end{theorem}

We can often verify part (b) of Proposition \ref{equivalence} by
using the results of \cite{bf} and \cite{bbf}.
In other cases, we can quote existing results in the literature
showing that $K$ is not Hilbert-Speiser of type $C_{p}$
(see \cite{gh}, \cite{herreng}).
Combining all of these results gives the following theorem.

\begin{theorem}\label{maintheorem}
Let $p$ be a prime and let $K$ be a number field.
Suppose that at least one of the following conditions holds:
\begin{enumerate}
\item $p=2$ or $3$,
\item $K/\Q$ is unramified at $p$,
\item $K$ is totally real, or
\item $K$ is (totally) imaginary and $K/\Q$ is Galois.
\end{enumerate}
Then  mHS-L$(K,C_{p})$ holds.
\end{theorem}

\begin{corollary} \label{Gal-mHS-L}
Let $p$ be a prime and let $K$ be a number field such that $K/\Q$
is Galois. Then  mHS-L$(K,C_{p})$ holds.
\end{corollary}

\begin{remark} Since the only primes of $K$ which divide $\tr_{L/K}(\OK)$ are those
which are wildly ramified in $L/K$, the conjectures mHS-L$(K,C_p)$
and HS-L$(K,C_p)$ coincide if every prime of $K$ above $p$ is
principal. In this case (for instance, if $p$ remains prime in $K$, or if $K$ has class number
1) one can replace ``mHS-L$(K,C_p)$'' with ``HS-L$(K,C_p)$'' in the conclusions of Theorem 1.10 and Corollary 1.11.
\end{remark}

We end the paper with some explicit examples. Firstly, we exhibit a
sextic field which is not Galois over $\Q$, and for which mHS-L$(K,C_5)$
(and also HS-L$(K,C_5)$) does not hold. This suggests that one should not
expect any significant strengthening of Theorem \ref{maintheorem} to
be possible. Secondly, we justify the introduction of the 
conjecture mHS-L$(K,G)$ in place of HS-L$(K,G)$ by giving several
examples of quartic fields $K$ satisfying condition (b) of 
Theorem \ref{maintheorem}, so that mHS-L$(K,C_5)$ holds, but for
which HS-L$(K,C_5)$ does \emph{not} hold.

\section{Local Failure for Elementary Abelian Extensions of Degree 4}

In this section, we prove Theorem \ref{C2C2fail} and Corollary
\ref{C2C2ex}. We first need a criterion which, for a $G$-Galois
extension $N/M$ of $p$-adic fields, guarantees that the valuation ring
$\ON$ of $N$ is \emph{not} free over its associated order
$\mathcal{A}_{N/M}$. Such a criterion is given by \cite[Theorem
3.13]{nonfree}. It is valid for any abelian $p$-group $G$, and in fact
applies not just to $\ON$ but to any power $\Frakp^h$ of the maximal
ideal $\Frakp$ of $\ON$. Specialising to the case $p=2$, $G=C_2 \times
C_2$, $N/M$ is totally ramified, and $h=0$, this result reads as
follows.
 
\begin{lemma} \label{non-free}
Let $M$ be a finite extension of $\Q_2$ with absolute ramification
index $e$. Let $N/M$ be a totally ramified $C_2 \times C_2$-Galois extension 
with ramification numbers $t_1 \leq t_2$ (in the lower numbering).
Suppose further that 
\begin{equation} \label{Vost}
  t_2 - \left \lfloor \frac{t_2}{2} \right \rfloor < 2e. 
\end{equation}
Let $w$ be the valuation of the different of $N/M$, and for $a \in \Z$
let $\overline{a}$ denote the least non-negative residue of $a$ modulo
$4$. Then $\ON$ is \emph{not} free over $\mathcal{A}_{N/M}$ if, for at least one
value of $i \in \{1,2\}$, we have
\begin{equation}\label{nonfree-crit}
  3 > \overline{t}_i > \overline{w}. 
\end{equation}
\end{lemma}

By Hilbert's formula for the different \cite[Chapter IV, \S1,
Proposition\ 4]{serre}, we have $w=3(t_1+1)+(t_2-t_1)$. Since $t_1$
and $t_2$ must both be odd \cite[Chapter IV, \S2, Proposition\ 11 and Exercise
3(f)]{serre}, it follows that
\begin{equation} \label{diff-val}
 \overline{w} = 3-\overline{t}_2. 
\end{equation}

We now turn to the question of recognising Hilbert-Speiser fields of
type $C_2 \times C_2$.
 
\begin{proposition}\label{HSC2C2}
Let $K$ be a number field such that $\Cl_4(\OK)$ is
trivial. Then $K$ is Hilbert-Speiser of type $C_2 \times C_2$.
\end{proposition}
\begin{proof}
For any finite group $G$, the ring of integers $\OL$ in a tame
$G$-Galois extension $L$ of $K$ is a locally free $\OK[G]$-module of rank
1. Moreover, if $G$ is abelian, then a locally free $\OK[G]$-module is
determined up to isomorphism by its rank and its class in the locally
free class group $\Cl(\OK[G])$. Thus, if $\Cl(\OK[G])$ is trivial, then
$\OL$ is necessarily free for any such $L$, and $K$ is Hilbert-Speiser
of type $G$.

We now take $G=C_2 \times C_2$. By \cite[\S2]{McCulloh} (or, more
explicitly, \cite[Proposition 2.4]{BS}), $\Cl(O_K[G])$ is isomorphic
to a certain quotient of 4 copies of $\Cl_4(O_K)$. Thus the triviality
of $\Cl_4(\OK)$ implies that of $\Cl(\OK[G])$.
\end{proof} 

\begin{proof}[Proof of Theorem \ref{C2C2fail}]
Let $K$ be as in the statement of the theorem, and let $G=C_2 \times
C_2$. Then $K$ is Hilbert-Speiser of type $G$ by Proposition
\ref{HSC2C2}.

Now let $\pi \in K$ be a local parameter at $\frakp$, let
$M=K_\frakp$ be the completion of $K$ at $\frakp$, and let $e \geq 3$
be the absolute ramification index of $\mathfrak{p}$. Let 
$E_1=K(\sqrt{1+\pi^{2e-5}})$, $E_2=K(\sqrt{\pi})$, and $L=E_1
E_2$. The completions of $E_1$ and $E_2$ at $\frakp$ are quadratic
extensions $F_1$, $F_2$ of $M$ with ramification numbers $u_1=5$,
$u_2=2e$ respectively. Using standard results on the upper and lower
ramification filtrations (see \cite[Chapter IV, \S3]{serre}), we find
that the completion of $L$ at $\frakp$ is the totally ramified
biquadratic extension $N=F_1 F_2$ of $M$ with \emph{upper}
ramification numbers $u_1<u_2$, and hence with lower ramification
numbers $t_1=u_1=5$ and $t_2=u_1+2(u_2-u_1)=4e-5$. We apply Lemma
\ref{non-free} to the $G$-Galois extension $N/M$. Firstly, (\ref{Vost}) holds
since $t_2 - \left \lfloor \frac{t_2}{2}\right \rfloor =
(4e-5)-(2e-3)=2e-2$. Secondly, (\ref{nonfree-crit}) holds for $i=1$ as
$t_1 \equiv 1 \pmod{4}$, and $w \equiv 0 \pmod{4}$ by
(\ref{diff-val}).  This shows that the valuation ring $\ON$ is not free
over $\mathcal{A}_{N/M}$. But since $\frakp$ is totally ramified in
$L/K$, we have $\mathcal{A}_{L/K,\frakp}=\mathcal{A}_{N/M}$. Hence
$\OL$ is not locally free over $\mathcal{A}_{L/K}$.
\end{proof}

\begin{proof}[Proof of Corollary \ref{C2C2ex}]
A list of all real cubic fields of discriminant $\leq 3132$ is given
by Cohen \cite[Table B.4]{cohen}. From this, we read off that the
fields $K=\Q(\alpha_i)$ for $i=1$, 2, 3 are real cubic fields: in
fact, they are the unique such fields (up to Galois conjugacy) of
discriminant 148, 404, 564, respectively, and none of them is normal over $\Q$. In each case,
$\OK=\Z[\alpha_i]$. By factoring the given polynomial over $\F_2$,
we check that 2 is totally ramified in $K$. Thus the unique prime of
$K$ above 2 has ramification index $e=3$.

It remains to verify that $\Cl_4(\OK)$ is trivial. This can be done
either using PARI \cite{PARI2}, or as follows. We observe that the
canonical map $\Cl_4(\OK) \twoheadrightarrow \Cl(\OK)$ is injective if
and only if the natural map $\OK^\times \longrightarrow
(\OK/4\OK)^\times$ is surjective. In each of the three cases, we can
check the surjectivity of the latter map by hand, using the
fundamental units given in Cohen's table. This table also tells us
that $K$ has class number 1. Thus $\Cl_4(\OK)=1$ as required.
\end{proof}

\section{Realisable Classes and the Proof of Proposition
  \ref{equivalence}}\label{realisable} 

Let $K$ be a number field and let $p$ be a prime.  Let $\Delta \cong
(\Z/p\Z)^{\times}$ be the group of automorphisms of $C_p$.  Then the
locally free class group $\Cl(\mathcal{O}_K[C_p])$ is a
$\Delta$-module.  As $L/K$ varies over all tame $C_{p}$-Galois
extensions of $K$, the class $(\mathcal{O}_L)$ of $\mathcal{O}_L$
varies over a subset $R(\mathcal{O}_K[C_p])$ of
$\Cl(\mathcal{O}_K[C_p])$. Let $\Cl(\mathcal{O}_K)$ denote the ideal
class group of $K$ and let $\Cl'(\mathcal{O}_K[C_p])$ be the kernel of
the map $\Cl(\mathcal{O}_K[C_p]) \longrightarrow \Cl(\mathcal{O}_K)$
induced by augmentation.
Let ${\mathcal J}$ be the Stickelberger
ideal in ${\mathbb Z}[\Delta]$. In \cite{McCulloh}, it is shown that
$R(\mathcal{O}_K[C_p])$ is the subgroup
$\Cl'(\mathcal{O}_K[C_p])^{\mathcal J}$ of $\Cl(\mathcal{O}_K[C_p])$
generated by $\{ c^{\alpha}: c \in \Cl'(\mathcal{O}_K[C_p]), \alpha \in {\mathcal J}\}$.

Now assume that $p$ is odd, and let $\Sigma$ be the element
$\Sigma_{g\in C_{p}}g$ in the group ring $K[C_{p}]$.  For any wildly
ramified $C_p$-Galois extension $L/K$, there is an integral
$\OK$-ideal $\frakb$ such that $\frakb(\mathcal{A}_{L/K}\cap
K\Sigma)=\OK\Sigma$, and it is easy to see that
${\tr}_{L/K}(\OL)={\frakb}$.  Miyata \cite{miyata} associates to $\OL$
a class $\mathrm{cl}(\OL)$ in the locally free class group
$\Cl(\mathcal{A}_{L/K})$, and then investigates the behaviour of
this class as $L$ varies over extensions with the same associated
order. Note however that in general $\OL$ need not be locally free
over $\mathcal{A}_{L/K}$, so $\mathrm{cl}(\OL)$ should not be
interpreted simply as ``the class of'' the $\mathcal{A}_{L/K}$-module
$\OL$. In the case that $\OL$ is locally free over $\mathcal{A}_{L/K}$,
one sees from \cite[ p.160]{miyata} that $\mathrm{cl}(\OL)$ is the
class in $\Cl(\mathcal{A}_{L/K})$ of the locally free module
$\frakb^{-1} \OL$, and not of $\OL$ itself. As in the previous
paragraph, $\Cl(\mathcal{A}_{L/K})$ is a $\Delta$-module and
\cite[Corollary to Theorem 2]{miyata} shows that $\mathrm{cl}(\OL)$
lies in $\Cl'(\mathcal{A}_{L/K})^{\mathcal J}$, where
$\Cl'(\mathcal{A}_{L/K})$ is defined analogously to
$\Cl'(\mathcal{O}_K[C_p])$ and ${\mathcal J}$ is again the
Stickelberger ideal in ${\mathbb Z}[\Delta]$. Moreover, there is a
surjective ${\Delta}$-homomorphism
\begin{equation}\label{surj-map}
f: \Cl'(\mathcal{O}_K[C_p])^{\mathcal J}\longrightarrow \Cl'(\mathcal{A}_{L/K})^{\mathcal J}.
\end{equation}

\begin{proof}[Proof of Proposition \ref{equivalence}]
If $K$ satisfies the modified Leopoldt condition of type $C_p$, then
it is clear that (a) and (b) hold once one recalls that
${\tr}_{L/K}(\OL)=\OK$, and that $\mathcal{A}_{L/K} =
\mathcal{O}_{K}[C_{p}]$, if and only if $L/K$ is a tame $C_{p}$-Galois
extension.

Suppose conversely that (a) and (b) both hold.  Since $K$ is a
Hilbert-Speiser field of type $C_p$, the subgroup
$\Cl'(\mathcal{O}_K[C_p])^{\mathcal J}$ is trivial by \cite[Theorem,
p.103]{McCulloh}. Now let $L/K$ be any wildly ramified Galois
extension with Galois group isomorphic to $C_p$. Since $\mathcal{O}_L$
is a locally free $\mathcal{A}_{L/K}$-module, it follows from the
discussion above that the class $\mathrm{cl}(\OL)=({\frak
b}^{-1}\mathcal{O}_L)$ lies in $\Cl'(\mathcal{A}_{L/K})^{\mathcal
J}$. However, $\Cl'(\mathcal{A}_{L/K})^{\mathcal J}$ is trivial by
(\ref{surj-map}), so ${\frak b}^{-1}\mathcal{O}_L$ is free over
$\mathcal{A}_{L/K}$. 
\end{proof}

By \cite[Theorem 2]{carter-arch}, we
know that a Hilbert-Speiser field $K$ of type $C_p$ must have
class number $h_K=1$ if either $p=2$, or $p=3$ and $K$ contains a primitive
cube root of unity $\zeta_3$. The condition $\zeta_3 \in K$ can be
removed by Lemma \ref{hK} below. But if $h_K=1$ then HS-L$(K,C_p)$ and
mHS-L$(K,C_p)$ coincide, so we obtain the following
additional corollary to Theorem \ref{maintheorem}.

\begin{corollary} \label{HSL23}
Let $p=2$ or $3$, and let $K$ be any number
  field. Then  HS-L$(K,C_{p})$ holds. 
\end{corollary}

We will see in \S\ref{section} below that there exist
Hilbert-Speiser fields of type $C_5$ having class number 2.

\begin{lemma}  \label{hK} 
If $K$ is Hilbert-Speiser of type $C_3$ and $\zeta_3 \not \in K$, then $h_K=1$.
\end{lemma}
\begin{proof}
Let $M=K(\zeta_3)$. Then $h_M=1$ by \cite[Theorem 2(ii)]{carter-arch}. If the quadratic extension
$M/K$ is ramified (either at a prime above 3 or at an infinite prime)
then the norm $\Cl(\OM) \longrightarrow \Cl(\OK)$ is surjective \cite[Theorem
10.1]{Was}, so $h_K=1$. If $M/K$ is everywhere unramified, then $K$ is
totally imaginary and each prime $\frakp$ of $K$ above $3$ has absolute
ramification index $e_\frakp \geq 2$. We claim that in this case $K$
cannot be Hilbert-Speiser of type $C_3$. 

To prove the claim, we apply Herreng's formula \cite[Proposition 3.2]{herreng}
for the 3-rank $d_3( (\OK/3\OK)^\times)$ of the unit group of the residue ring
$\OK/3\OK$. Writing $f_\frakp$ for the inertia degree of $\frakp$, this yields 
$$
 d_3( (\OK/3\OK)^\times) 
   =  \sum_{\frakp | 3} f_\frakp \left( e_\frakp - 1 - 
   \left\lceil \frac{e_\frakp-3}{3}   \right\rceil \right). 
$$
To show $K$ is not Hilbert-Speiser of type $C_3$, it suffices by
\cite[Theorem 1]{grrs} to show that 
$V_3(K):=(\OK/3\OK)^\times / \mathrm{im}(\OK^\times)$ has exponent divisible by
3. This will certainly hold if 
\begin{equation} \label{d3rank}
    d_3( (\OK/3\OK)^\times) > d_3(\OK^\times) = \frac{1}{2}[K:\Q] - 1.
\end{equation}
(The equality holds since $K$ is totally imaginary and $\zeta_3 \not
\in K$.) But we calculate
\begin{eqnarray*}
\lefteqn{2\biggl( d_3( (\OK/3\OK)^\times)- d_3(\OK^\times) \biggr)  } \\
  & = & 2 \sum_\frakp f_\frakp \left( e_\frakp - 1 - \left\lceil
  \frac{e_\frakp-3}{3} \right\rceil \right) -  \left( \sum_\frakp f_\frakp
  e_\frakp -2 \right) \\    
   & = &  \sum_\frakp f_\frakp \left( e_\frakp-2 -2 \left \lceil
  \frac{e_\frakp-3}{3}  \right \rceil \right)  + 2 \\
 & \geq & 2,  
\end{eqnarray*}
since $e_\frakp \geq 2$ for each $\frakp$. Hence (\ref{d3rank})
  holds, as required.
\end{proof}

\section{Local Freeness for $C_p$-Galois extensions}

Let $p$ be prime and let $N/M$ be a wildly ramified Galois extension of $p$-adic fields of degree $p$. In this section, we briefly review the results of \cite{bf} and \cite{bbf}, which give necessary and sufficient conditions for $\mathcal{O}_{N}$ to be free over $\mathcal{A}_{N/M}$ in this case.

Let $e$ denote the ramification index of $M/\Q_{p}$ and let $t$ denote the ramification number of $N/M$.  From \cite[Chapter IV, \textsection 2, Exercise 3]{serre} and the assumption that 
$N/M$ is wildly ramified, we have $1 \leq t \leq \frac{pe}{p-1}$. 
Define $n$ to be the ``length'' of the continued fraction expansion
\[ 
\frac{t}{p} = a_{0} + \frac{1}{a_{1}+ \frac{1}{a_{2 + \ldots}} } = [a_{0}, a_{1}, a_{2}, \ldots, a_{n}],
\quad \textrm{ with } a_{n}>1.
\]
Let $a$ be the unique integer such that $0 \leq a \leq p-1$ and $a \equiv t \, \mod \, p$
(and so $t=a_{0}p+a$). 

\begin{theorem}[\cite{bf}, \cite{bbf}]\label{local-criteria}
The ring of integers $\mathcal{O}_{N}$ is free over $\mathcal{A}_{N/M}$ if and only if
\begin{enumerate}
\item $a=0$, or
\item $t < \frac{pe}{p-1}-1$ and $a \mid (p-1)$, or
\item $\frac{pe}{p-1}-1 \leq t$ and $n \leq 4$.
\end{enumerate}
\end{theorem}

\begin{corollary}\label{3-corollary}
If $p=3$, then $\mathcal{O}_{N}$ is always free over $\mathcal{A}_{N/M}$.
\end{corollary}

\begin{proof}
Since $p=3$, we must have $a \in \{ 0,1,2 \}$.
The case $a=0$ is clear, so suppose $a \in \{ 1,2 \}$.
Then $a$ divides $2=p-1$ and either 
\[ 
\frac{t}{p} = a_{0} + \frac{1}{3} 
\quad \textrm{ or } \quad 
\frac{t}{p} = a_{0} + \frac{2}{3} = a_{0} + \frac{1}{1 + \frac{1}{2}},
\]
so $n \leq 2$. Since either $t < \frac{pe}{p-1}-1$ or $\frac{pe}{p-1}-1 \leq t$,
this completes the proof.
\end{proof}

We end this section with a global consequence of Theorem
\ref{local-criteria}.

\begin{corollary}\label{bp}
Let $p \geq 5$ be prime, and let $b(p)$ be the least positive
integer not dividing $p-1$. Let $K$ be a number field such that some
prime $\frakp$ of $K$ above $p$ has absolute ramification index $e >
b(p)$. Then there is a $C_p$-Galois extension $L$ of $K$ such that
$\OL$ is \emph{not} locally free over $\mathcal{A}_{L/K}$.
\end{corollary}
\begin{proof}
Let $M=K_\frakp$ be the completion of $K$ at $\frakp$. For any integer
$t$ such that $0 < t < pe/(p-1)$ and $t \not \equiv 0 \pmod{p}$, there
is a (totally ramified) $C_p$-Galois extension $N$ of $M$ with ramification
number $t$: see for instance \cite[Chapter III, (2.5) Proposition]{FV}. We
take $t=b(p)$. As $p \geq 5$, we have $0<t<p-1$, so that $t \not
\equiv 0 \pmod{p}$. Also, $t<\frac{pe}{p-1} - 1$ since $e \geq t+1$ by
hypothesis. Thus there exists a Galois $C_p$-extension $N$ of $M$ with
ramification number $t$, and moreover $\ON$ is not free over
$\mathcal{A}_{N/M}$ by Theorem \ref{local-criteria} and the definition
of $b(p)$.

Having found the $p$-adic extension $N/M$, we observe
that there exists a $C_p$-Galois extension $L$ of the number
field $K$ such that $L_\frakp =N$: this follows from the
Grunwald-Wang theorem (see for example \cite[(9.2.3) Corollary]{NSW}; note
we are not in the ``special case'' since $p$ is odd). Then $\OL$ is
not locally free over $\mathcal{A}_{L/K}$, since $\ON$ is not free
over $\mathcal{A}_{N/M}$. 
\end{proof}

\section{Proof of Theorem \ref{maintheorem}}

We now combine the results of previous sections and of \cite{carter}, \cite{herreng}
to prove our main result.

\begin{proof}[Proof of Theorem \ref{maintheorem}]
Assume that $K$ is a Hilbert-Speiser field of type $C_p$. 
By Proposition \ref{equivalence}, it suffices to show that if $L/K$ is a wildly ramified 
$C_{p}$-Galois extension, then $\mathcal{O}_L$ is a locally free $\mathcal{A}_{L/K}$-module. To this end let $M=K_{\mathfrak{p}}$ be the completion of $K$ at some prime 
$\mathfrak{p}$ above $p$ and let $e$ be the ramification index of $M/\Q_{p}$.
Let $N/M$ be a $C_{p}$-Galois extension and assume that $N/M$ is (wildly) ramified. 
Let $t$ be the ramification number of $N/M$.  
From \cite[Chapter IV, \textsection 2, Exercise 3]{serre} and the assumption that 
$N/M$ is wildly ramified, we have $1 \leq t \leq \frac{pe}{p-1}$. 

We always have $\zeta_{2}=-1 \in K$, so the case $p=2$ is given by 
\cite[Theorem 1.2]{carter}. The case $p=3$ follows from Corollary \ref{3-corollary}.
So we may henceforth assume that $p \geq 5$.

Suppose that $K/\Q$ is unramified at $p$. 
Then $e=1$ and so the above inequality becomes $1\leq t \leq \frac{p}{p-1}$, 
which forces $a=t=1$. It now follows from 
Theorem \ref{local-criteria} (c) that $\mathcal{O}_N$ is a free $\mathcal{A}_{N/M}$-module.
So we may henceforth assume that $e \geq 2$.

Suppose that $K$ is totally real. By Theorem \ref{real-not-HS}, we are reduced to the
case $p=5$ and $e=2$. We must have $a=t \in \{ 1,2 \}$. If $t=1$, the result follows
from Theorem \ref{local-criteria} (b); if $t=2$, it follows from Theorem \ref{local-criteria} (c).

Suppose that $K$ is (totally) imaginary and $K/\Q$ is Galois.
By the end of the proof of \cite[Proposition 3.4]{herreng}, we
must have $e \leq \frac{2p}{p-2}$, since otherwise $K$ would not be 
Hilbert-Speiser of type $C_{p}$. The inequalities $2 \leq e \leq \frac{2p}{p-2}$
and $1\leq t \leq \frac{pe}{p-1}$ now leave several cases to consider.
Since $p \geq 5$, we have $\frac{2p}{p-2}< 4$, so  $e \in \{2,3\}$. If 
$e=2$ then $a=t \in \{1,2\}$. If $t=1$, the result follows from Theorem \ref{local-criteria} (b); 
if $t=2$, it follows from Theorem \ref{local-criteria} (c) since
\[
\frac{t}{p} = \frac{2}{p} = \frac{1}{\left( \frac{p-1}{2} \right)+ \frac{1}{2}}.
\]
Now assume $e=3$. 
In this case we find $a=t \in \{1,2,3\}$. When $t \in \{1,2\}$ the result follows from Theorem \ref{local-criteria} (b). If $t=3$ then either
\[
\frac{t}{p} = \frac{3}{p} = \frac{1}{\left( \frac{p-1}{3} \right)+ \frac{1}{3}}
\quad
\textrm{ or }
\quad
\frac{t}{p} = \frac{3}{p} = \frac{1}{\left( \frac{p-2}{3} \right)+ \frac{1}{1 + \frac{1}{2}}},
\]
and so the result follows from Theorem \ref{local-criteria} (c).
\end{proof}

\section{Counterexamples for $C_5$}

Most of this section is devoted to providing a counterexample to 
mHS-L$(K,C_5)$ in the case that $K/\Q$ is not Galois. In this
counterexample, mHS-L$(K,C_5)$ fails locally. 
In the last part we give some examples of fields $K$ such that
mHS-L$(K,C_5)$ holds, but HS-L$(K,C_5)$ does not; the failure of
HS-L$(K,C_5)$ is then necessarily a genuinely global phenomenon. 

The counterexample will be a sextic field $K$ over the rationals with
signature $(2,2)$. The defining
polynomial is $x^6+2x^4-5x-5$.
 The field $K$ has class
number 1, and the prime 5 splits in $K$ as
the fourth power of a degree one prime $\frakp_1$ times a degree two  prime
$\frakp_2$. Applying Corollary \ref{bp}, and noting that $b(5)=3$, it
follows that $K$ has a $C_5$-Galois extension $L/K$ for which 
$\OL$ is not locally free over $\mathcal{A}_{L/K}$.

It is therefore left to prove that $K$ is Hilbert-Speiser of type $C_5$.
To do this, we must show that the subgroup of realisable classes
in the class group of $\OK[C_5]$ is trivial. 

Since 5 ramifies in $K$, the calculation of the class group of $\OK[C_5]$ is a little difficult,
involving two fibre products, but not as difficult as one might expect since some relevant
ray class groups are trivial and of order 2, respectively. The final outcome is that the
class group of $\OK[C_5]$ is of  order 1 or 2. This
implies, by an easy explicit argument, that $\Cl(\OK[C_5])$ is
annihilated by the Stickelberger ideal (note that the 
cyclic group Aut$(C_5)$ of order 4 has to act trivially). Therefore
$K$ is Hilbert-Speiser of type $C_5$, by McCulloh's theorem on
realisable classes in \cite{McCulloh}. In fact, with extra effort one can 
show that the order of the whole class group is 1; but it is much simpler to use
McCulloh's theorem.

 \subsection{The class group of a certain nonmaximal order $R$ in $M=K(\zeta_5)$}

Let $K = \Q(\theta)$ with $\theta^6+2\theta^4-5\theta-5=0$. 
We check using PARI \cite{PARI2} that $5\OK = \frakp_1^4\frakp_2$, where $\frakp_1$ has degree 1
and $\frakp_2$ has degree 2. We let $M=K(\zeta_5)$. Calculation shows
that $\frakp_1\OM$ is  again prime. 

\begin{remark} $M$ has very large degree (24),
but the polynomial defining $M$, afforded by the PARI command {\tt polcomposite}, is not
as unwieldy as one might expect. We then calculated {\tt bnfinit} of $M$, which
contains all the information we need on (ray) class groups of $M$.

Some remarks on the length and the reliability of our calculations:  The algorithm {\tt bnfinit} took seconds or minutes, depending on the choice of governing parameters.
We used the parameters $c=0.3$, $c2=12$, which yield a rigorous
result under the assumption GRH. (It seems illusory to
eliminate GRH for a field of this size, in particular
a call of {\tt bnfcertify} results in an instantaneous
refusal, because the Minkowski constant is much too large.) The inbuilt
check number was 1, as it should be.
Some more
plausibility checks were done, such as repeating the calculation with another
defining polynomial afforded by {\tt polred}, or on
different machines. 
 
Another partial justification of correctness
is as follows. All the units produced by \cite{PARI2}  were double-checked
(a quick way is to take the principal
ideal generated by a hypothetical unit and factor it; if
one gets the empty factorisation, we indeed have a unit). 
Now if we accept the statement that the class number $h_M$ of $M$ is 1 (produced by {\tt bnfinit}) as true, then the triviality of  a certain ray class group (established below)
is rigorously true as well: 
The unlikely case
that PARI missed some units of $M$ would only mean 
that our number for the order of the ray class group
might be too high, but we already obtain order 1 using the
supply of units found by \cite{PARI2}. 
\end{remark}

The main task in this section is a comparison of the 
rings $R=\OK[\zeta_5] \cong \OK\otimes_{\Z}\Z[\zeta_5]$,
which is {\it not\/} the maximal order of $M$, on the one side,
and the ring $S=\OM$ on the other side. By \cite{PARI2} $S$ has
class number 1. We want to establish that $\Cl(R)$ is trivial as well.

For this we  need some analysis of the inclusion $R \subset S$.
Since $S$ disagrees with $R$ at most at primes which are ramified
both in $K$ and in $\Q(\zeta_5)$,
 the only prime at which we expect disagreement is $\frakp_1$. 
Let $\pi$ be a local parameter at $\frakp_1$ in $K$,
and let $\lambda=\zeta_5-1$.
Since 5 is (tamely) ramified with ramification index 4 both in $K_{\frakp_1}$ and 
in $\Q_{5}(\zeta_5)$, the rings
$\O_{M,\frakp_1}$ and $\Z_5[\pi,\lambda]$ cannot be equal. But one has the following
local description at $\frakp_1$:

\begin{lemma}  The element $\xi: = \lambda/\pi$ is integral. The ring $S_{\frakp_1}=\O_{M,\frakp_1}$ is
  the (free) $\O_{K,\frakp_1}$-span of $1,\xi,\xi^2, \xi^3$,
  and the length of $\O_{M,\frakp_1}/\Z_5[\pi,\lambda]$ is 6. 
  Moreover, $\pi^3 S_{\frakp_1}$ is contained in $R_{\frakp_1}=\Z_5[\pi,\lambda]$.
  \end{lemma}

\begin{proof}
The integrality of $\xi$ is clear, by looking at valuations. 
(Recall that $\pi$ is also a parameter for the extended ideal $\frakp_1\O_M$.)
The $\O_K$-module $T$ defined as the $\O_K$-span of $1,\xi,\xi^2, \xi^3$
has the easily seen property that $T/\Z_5[\pi,\lambda]$
is of length 6 $(=1+2+3)$ over $\O_K/\frakp_1$, that is, of order $5^6$. By 
comparing the discriminants of the algebras $R$ and $S$ (given by \cite{PARI2} and
Schachtelungsformel), one sees that $\O_{M,\frakp_1}/\Z_5[\pi,\lambda]$ also
has order $5^6$. Hence we have equality. The last statement follows
from the definitions. 
\end{proof}

\begin{corollary} We also have the global inclusion $\frakp_1^3 S \subset R$.
\end{corollary}

This corollary produces a fibre product arising as follows. 
If  we let $\bar S$ stand for $S/\frakp_1^3 S$ and $\bar R$ for the image of $R$ in $\bar S$,
we obtain the fibre product 
\[
\begin{array}{ccc}
R & \hookrightarrow & S \\
\downarrow &  & \downarrow \\
\bar{R}& \hookrightarrow &\bar{S}
\end{array}
\]
where the horizontal arrows are the natural inclusions and the vertical arrows are the natural projections. By \cite[Theorem 42.13]{curtis-reiner2} and the fact that all rings are commutative we have the following exact sequence
\[
S^{\times} \times {\bar{R}}^{\times} 
\longrightarrow 
{\bar{S}}^{\times} 
\longrightarrow
\Cl(R)
\longrightarrow
\Cl(S)\oplus\Cl(\bar{R}).
\]
Since  $\bar R$ 
has trivial class group (being semilocal), and since
the class group of $S$ is also trivial, the class group $\Cl(R)$ 
 is an epimorphic 
image of $U:={\bar S}^{\times} /({\bar R}^{\times}  \cdot {\rm im}(S^{\times} ))$. It thus suffices
to establish that $U$ is trivial. In fact, we will
check that $U':={\bar S}^{\times} /{\rm im}(S^{\times} )$ is already trivial. But $U'$ is precisely
the ray class group modulo $\frakp_1^3$ of $M$.  By \cite{PARI2} we find 
that this ray class group is trivial, so we are done. (It is
perhaps interesting to mention that the ray class number
of $M$ modulo $5\O_M$ is relatively large, being equal to $15625=5^6$.)

We sum up: $R = \O_K \otimes_{\Z} \Z[\zeta_5]$ has class number 1. We note that, in principle, this calculation could also be performed using the algorithm of \cite{kluners-pauli}.

\subsection{The class group of the integral group ring}

To calculate the class group of $\O_K[C_5]$, we now have to look at a second fibre product
with $\O_K[C_5]$ at the upper left; the upper right and lower left corners are occupied
by $R= \O_K \otimes_\Z \Z[\zeta_5]$ and $\O_K$, respectively. Finally, the lower right
hand corner has the ring $T = R/\lambda R = \O_K/5\O_K$. (Recall $\lambda=1-\zeta_5$.)
The class group of $R$ is trivial
(see previous section), and so is the class group of $\O_K$ (\cite{PARI2}). 
Hence, applying \cite[Corollary 49.28]{curtis-reiner2}, we find that  the class group of $\O_K[C_5]$ is an epimorphic image of 
the group $X:=T^\times/({\rm im}(R^\times)\cdot {\rm im}(\O_K^\times))$. Our final claim 
 will follow, if we can establish that $X$ is
of order 1 or 2. But, reasoning as above, $X$ is an epimorphic
image of $T^\times/{\rm im}(\O_K^\times)$, the ray class group of $K$ modulo $5\O_K$.
By \cite{PARI2} this group has order 2 . This is safe,
since we ran {\tt bnfcertify} on $K$, which confirmed the
whole output of {\tt bnfinit}. We mention in passing that the fundamental units of $K$
have remarkably small coefficients. Let us also remark that we cannot expect ray
class number 1 here, because of the existence of the quadratic extension $K(\sqrt 5)/K$.
(Getting back to a remark made at the beginning: it is possible, but not
easy, to show that $X$ is of order 1. We also note that, in principle, this calculation can be performed using the algorithm of \cite{bley-boltje}.)

This completes the proof that the class group
of $\O_K[C_5]$ has order at most 2, and this implies,
as explained at the beginning, that $K$ is Hilbert-Speiser of type $C_5$.

\subsection{Some examples and more counterexamples}\label{section}

In this section we give some examples of fields $K$ such that $K$ has
class number 2 and is Hilbert-Speiser of type $C_5$ (as mentioned
after Corollary \ref{HSL23}). The examples $K$ will be quartic fields
over the rationals with signature $(2, 1)$. Moreover, $K/\Q$ will be
unramified at 5, and $K$ will contain a nonprincipal prime $\frak p$
above 5. By Theorem \ref{maintheorem} (b), mHS-L$(K,C_5)$
holds. However, there exists a $C_5$-Galois extension $L/K$ with
${\tr}_{L/K}(\OL)={\frak p}$. This follows from the Grunwald-Wang
theorem as in the proof of Corollary \ref{bp}, but we construct one
such $L$ explicitly in the Appendix to this section. The existence of
such an extension $L$ shows that HS-L$(K,C_5)$ does not hold: The
class $({\frak p}^{-1}\mathcal{O}_L)$ in $\Cl(\mathcal{A}_{L/K})$ is
trivial by mHS-L$(K,C_5)$, so the class $(\mathcal{O}_L)$ in
$\Cl(\mathcal{A}_{L/K})$ cannot be trivial.

In order to produce our fields $K$ we used \cite{PARI2} and \cite{QaoS} to generate lists of degree 4 polynomials $f_i$ with coefficients of absolute value less than or equal to 20, and then successively sifted out along the following criteria: 
\begin{enumerate}
\item $f_i$ irreducible
\item the field $K_i$ defined by $f_i$ has signature (2,1)
\item 5 is unramified but not inert in $K_i$
\item $h_{K_i}=2$
\item at least one prime of $K_i$ above 5 is nonprincipal
\item the field $M_{i}=K_i(\zeta_5)$ has class number 2 which equals the ray class number of $M_i$ modulo $(1-\zeta_5)$.
\end{enumerate}
In fact, using \cite{QaoS}, one quickly obtains 12051 polynomials satisfying (a), (b), and (d). With some extra effort utilising \cite{bcp} we then found that  $x^4-x^3+3x^2-3x-4$ satisfies all the stated criteria. This example was  then successfully checked using \cite{PARI2} on another machine.

Let $K=\Q(\theta)$ with $\theta^{4}-\theta^{3}+3\theta^{2}-3\theta-4=0$. Since 5 does not ramify in
$K$, $K$ is arithmetically disjoint to $\Q(\zeta_5)$, and hence
$\O_K \otimes_{\Z} \Z[\zeta_5]$ is the ring of integers in $M=K(\zeta_5)$.
Since the ray class group of $M$ modulo $(1-\zeta_5)$ has order 2,
and $h_M=2$, it follows that the quotient
$U':=(\mathcal{O}_M/(1-\zeta_5))^{\times}/{\rm im}(\mathcal{O}_{M}^{\times})$ is
trivial.

We have a fibre product
with $\mathcal{O}_K[C_5]$ at the upper left, and with the upper right and lower left corners occupied
by $\mathcal{O}_M= \mathcal{O}_K \otimes_\Z \Z[\zeta_5]$ and $\mathcal{O}_K$, respectively.
Finally, the lower right hand corner has the ring $\mathcal{O}_M/(1-\zeta_5)$. 
By \cite[Theorem 42.13]{curtis-reiner2} and the fact that all rings are commutative we obtain an exact sequence which appears as the middle row in the following diagram 
\[
\begin{array}{ccccccccc}
& & & &0 & &0\\
& & & & \downarrow &  & \downarrow \\
& & & & \Cl'(\O_K[C_5])&  & \Cl(\O_M) \\
& & & & \downarrow &  & \downarrow \\
1&\rightarrow&U&\rightarrow&\Cl(\O_K[C_5])&\rightarrow& \Cl(\O_M)\oplus\Cl(\O_K),\\
& & & & \downarrow &  & \downarrow \\
& & & &\Cl(\O_K)&=&\Cl(\O_K)\\
& & & & \downarrow &  & \downarrow \\
& & & &0 & &0
\end{array}
\]
where the left and right vertical sequences are, respectively, the
exact sequence of \S3, and the exact sequence defined in terms of the
natural maps. One easily verifies that this diagram is commutative so
there is a unique map $\alpha:\Cl'(\O_K[C_5])\rightarrow \Cl(\O_M)$
which makes the resulting diagram commute. Applying the snake lemma to
the vertical sequences and maps between them we obtain the exact
sequence

\[
0
\longrightarrow 
\mathrm{ker}(\alpha)
\longrightarrow
U
\longrightarrow
0.
\] 
Hence $\alpha$ has kernel $U$. That is, the following sequence is exact
\[
1
\longrightarrow 
U
\longrightarrow
\Cl'(\O_K[C_5])
\stackrel{\alpha}{ \longrightarrow} 
\Cl(\O_M).
\]
Since $U=U'/{\rm im}(\O_K^{\times})$ is trivial, $\Cl'(\O_K[C_5])$ maps injectively into $\Cl(\O_M)$, and hence is of order 1 or 2. So it is annihilated by the Stickelberger ideal, and $K$ is Hilbert-Speiser of type $C_5$. Furthermore, as explained above, mHS-L$(K,C_5)$ holds, but  HS-L$(K,C_5)$ does not.

Using the
techniques described above, we can find many more examples of such
polynomials.  We list three here, including our example above,  along with  the corresponding discriminant $d_K$:
\[
\begin{array}{llllll}
x^4-x^3+3x^2-3x-4 & & & (d_K=-12844=-2^2\cdot13^2\cdot19)\\
x^4-x^3+4x^2-4x-1 & & & (d_K=-17051=-17^2\cdot59)\\
x^4-2x^3+3x^2+x-4 & & & (d_K=-17231, {\rm which\  is\  prime}).
\end{array}
\]

Appendix: 

A $C_5$-Galois extension $L/K$ with ${\tr}_{L/K}(\OL)={\frak p}$ for the first example $K$
in the short list at the end of \S6 can actually be constructed explicitly. From 
\cite{MM} we find a polynomial 
\begin{eqnarray*}
 g(x,t) &=& x (x^2-25)^2+(x^4-20 x^3-10 x^2+300 x-95) t^2-4 (x-3)^2 t^4 \\
   &=& x^5 + t^2 x^4 + (-20 t^2 - 50) x^3 + (-4 t^4 - 10 t^2) x^2  \\
  & & \ + (24 t^4 + 300 t^2 + 625) x + (-36 t^4 - 95 t^2)
\end{eqnarray*}
whose Galois group over $\Q(t)$ is the cyclic group $C_5$ of order 5. Whenever $\tau\in K$
is such that $g_\tau(x) = g(x,\tau)$ is irreducible over $K$, the specialised polynomial
$g_\tau$ has Galois group $C_5$ over $K$. With some guesswork it is now possible to find
a choice of $\tau$ such that the resulting extension $L=L_\tau$ is ramified at the degree
one prime over 5 in $K$ and unramified at the other prime over 5. We only give the outcome;
everything has been verified by PARI, using a variety of double-checks, such as calculating discriminants
both for the relative extension $L/K$ and the absolute field $L$, and comparing via the
Schachtelungsformel. Also the a priori fact
that $L/K$ is cyclic (which would be difficult to prove by PARI in full rigour) was
tested by looking at the factorisation of several dozen prime ideals. Let $\theta$ denote
a root of the defining polynomial for $K$, so $\theta^4-\theta^3+3 \theta^2-3 \theta -4 =0$.
We are now free again to use $x$ as the variable for the polynomial defining $L/K$. We find
\begin{eqnarray*} 
g_\tau &=& x^5 + (-9 \theta^3 + 6 \theta^2 + 15 \theta + 8) x^4 + (180 \theta^3 - 120 \theta^2 - 300 \theta - 210) x^3 \\
          & & \ +(666 \theta^3 - 8364 \theta^2 + 5370 \theta + 7728) x^2 + (-6156 \theta^3 + 51624 \theta^2 - 28620 \theta - 43823) x  \\
          & & \ + (6039 \theta^3 - 75306 \theta^2 + 48255 \theta + 69512). 
\end{eqnarray*}
Apart from the wild ramification at one prime above 5, $L/K$ has tame ramification in two 
primes of norm 66821 and 4268881 respectively. (These two numbers are prime.)  

\section{Acknowledgments}

Carter would like to thank the members of Laboratoire A2X of the Universit{\'e} de Bordeaux I, and the members of the Mathematics Department of the Universit\"at der Bundeswehr M\"unchen for their kind hospitality during his visits in the 2006--2007 academic year.

Johnston would like to thank the Deutscher Akademischer Austausch Dienst (German Academic Exchange Service) for a grant allowing him to visit Greither at Universit\"at der Bundeswehr M\"unchen
for the 2006--2007 academic year, and his hosts for making his stay most productive and enjoyable.

\end{document}